\documentclass[leqno,a4paper]{article}
\usepackage[english]{babel}
\usepackage{amsmath}
\usepackage{amsthm}
\usepackage{amssymb}
\usepackage{enumerate}
\usepackage{xspace}
\usepackage{euscript}
\usepackage{graphicx}
\usepackage{graphics}
\usepackage{amscd}
\usepackage{epsfig}
\usepackage{tabularx}
        \newtheorem{lemma}{Lemma}[section]
        \newtheorem{proposition}[lemma]{Proposition}
        \newtheorem{theorem}[lemma]{Theorem}
        \newtheorem{definition}{Definition}[section]
        \newtheorem{remark}[lemma]{Remark}

\title{\bf{Lipschitz stability for the electrostatic inverse boundary value problem with piecewise linear conductivities}}
\author{Giovanni Alessandrini\thanks{Dipartimento di Matematica e Geoscienze, Universit\`{a} di Trieste, Italy. Email:alessang@units.it}\qquad{Maarten V. de Hoop\thanks{Departments of Computational and Applied Mathematics, Earth Science, Rice University, Houston, Texas, USA. Email:mdehoop@rice.edu}}\qquad\\
Romina Gaburro\thanks{Department of Mathematics and Statistics, University of Limerick, Ireland.  Email: romina.gaburro@ul.ie}\qquad
Eva Sincich\thanks{Dipartimento di Matematica e Geoscienze, Universit\`{a} di Trieste, Italy. Email:esincich@units.it}}

\date{\today}

\begin{document}
\maketitle

\begin{abstract}
We consider the electrostatic inverse boundary value problem also known as electrical impedance tomography (EIT) for the case where the
conductivity is a piecewise linear function on a domain
$\Omega\subset\mathbb{R}^n$ and we show that a Lipschitz stability
estimate for the conductivity in terms of the local
Dirichlet-to-Neumann map holds true.
\end{abstract}

\section{Introduction}\label{sec1}

We consider the inverse boundary value problem (IBVP) associated with
the elliptic equation for an electric potential, where the objective
is to recover electrical resistivity, or conductivity, from partial
data.  We focus our attention on the stability of this inverse
problem, in particular, when the conductivity is isotropic. We obtain
a Lipschitz stability result if the conductivity is known to be
piecewise linear on a given domain partition.

We let $\Omega$ be a bounded domain in $\mathbb{R}^n$, $n \geq 2$. In
the absence of internal sources, the electric potential, $u$,
satisfies the elliptic equation
\begin{equation}\label{eq conduttivita'}
   \operatorname{div}(\gamma\nabla{u}) = 0\qquad
\text{in}\quad\Omega ,
\end{equation}
where the function $\gamma$ signifies the \textit{conductivity} in
$\Omega$; $\gamma$ is a bounded measurable function satisfying the
ellipticity condition,
\begin{equation}\label{ellipticity}
   0 < \lambda^{-1} \leq \gamma \leq \lambda ,\qquad
\textnormal{almost\:everywhere\:in}\:\Omega ,
\end{equation}
for some positive $\lambda \in \mathbb{R}$. The inverse conductivity
problem consists of finding $\gamma$ when the so-called
Dirichlet-to-Neumann (DtoN) map
\begin{equation}\label{DtoN}
   \Lambda_{\gamma} :\ u\vert_{\partial\Omega}
         \in H^{\frac{1}{2}}(\partial\Omega)
   \longrightarrow \gamma \nabla{u} \cdot
   \nu\vert_{\partial\Omega} \in H^{-\frac{1}{2}}(\partial\Omega)
\end{equation}
is given for any weak solution $u \in H^1(\Omega)$ to \eqref{eq
  conduttivita'}. Here, $\nu$ denotes the unit outward normal to
$\partial\Omega$. If measurements can be taken on a portion $\Sigma$
of $\partial\Omega$ only, then the relevant map is referred to as the
local DtoN map.

This inverse problem has different appearances, namely, as electrical
impedance tomography (EIT) and direct current (DC) method or
electrical resistivity tomography (ERT) in geophysics (belonging to
the class of potential field methods). Although the mathematical
framework of this paper is the one described by \eqref{eq
  conduttivita'}-\eqref{DtoN}, the application we have in mind is the
determination of the resistivity $\rho = \gamma^{-1}$ in DC or ERT
methods corresponding with the following type of experiment or
``sounding'': a current is injected into the ground through a pair of
electrodes at the boundary while the voltage is measured with another
pair of electrodes. Thus the data, viewed as an operator, can be
identified with the so-called Neumann-to-Dirichlet (NtoD) map. We note
that in the mathematical literature the use of the DtoN map as the
data is more common. The DtoN map is invertible on its range. Indeed
the applied boundary current fluxes must have a vanishing average. We
note that the solution is defined up to an additive (grounding
potential) constant. Whereas it is well known that, at a theoretical
level, the knowledge of either of the two maps is equivalent, matters
may be more complicated when, in different applications, the physical
settings provide different discrete and noisy samples of such
maps. The NtoD map, upon applying it to a particular subset of
currents, provides the so-called \textit{apparent resistivity}. To be
precise, the apparent resistivity is a geometrical (acquisition)
factor multiplied by the ratio of voltage (potential difference) over
current.

The first mathematical formulation of this inverse problem is due to
Calder\'{o}n in the context of EIT in \cite{C}, where he addressed the
problem of whether it is possible to determine the (isotropic)
conductivity from the DtoN map. To be precise, Calder\'{o}n investigated the injectivity of the map

\[Q:\gamma\longrightarrow Q_{\gamma},\]

where $Q_{\gamma}(\phi)$ is the quadratic form associated to $\Lambda_{\gamma}$, by linearizing the problem.
As main contributions in this respect we mention the papers by Kohn and
Vogelius \cite{Koh-V1,Koh-V2}, Sylvester and Uhlmann \cite{Sy-U}, and
Nachman \cite{N}. We wish to recall the uniqueness results of Druskin
who, independently from Calder\'{o}n, dealt directly with the
geophysical setting of the problem in \cite{D1}-\cite{D3}. We also
refer to \cite{Bo}, \cite{Che-I-N} and \cite{U} for an overview of
recent developments regarding the issues of uniqueness and
reconstruction of the conductivity.

It is well known that the IBVP of determining the conductivity
$\gamma$ from the DtoN map is ill-posed. Indeed, regarding the
stability of this inverse problem, Alessandrini \cite{A} proved that,
assuming $n \geq 3$ and \textit{a-priori} bounds on $\gamma$ of the
form
\begin{equation}\label{regularity bounds on sigma}
   || \gamma ||_{H^s(\Omega)} \leq E ,
\quad\textnormal{for\:some}\:s>\frac{n}{2}+2 ,
\end{equation}
$\gamma$ depends continuously on $\Lambda_{\gamma}$ with a modulus of
continuity of logarithmic type. We also refer to \cite{A1}, \cite{A2},
which improve the result in \cite{A} for conductivities $\gamma \in
W^{2,\infty}(\Omega)$. We refer to \cite{B-B-R}, \cite{B-F-R} and
\cite{Liu} for the two-dimensional case, where logarithmic type
stability estimates have been established too. The common logarithmic
type of stability cannot be avoided \cite{A3, Ma}. However, the ill-posed
nature of this problem can be modified to be conditionally well-posed
by restricting the conductivity to certain function
subspaces. Well-posedness is here expressed by Lipschitz stability.

A first result of this kind was established by Alessandrini and
Vessella \cite{A-V}, to which we refer, together with \cite{A3}, for
an in-depth description and analysis. The result of \cite{A-V} was
extended to different types of problems, for example, in
\cite{Be-dH-Q}, \cite{Be-dH-Q-S} for the Schr\"odinger and the
Helmholtz equations, respectively, in \cite{Be-Fr} for the inverse
conductivity problem with complex conductivity, and in \cite{Be-Fr-V},
\cite{Be-Fr-Mo-Ro-Ve} for the determination of the Lam\'e parameters
in the elastostatic problem. All of these papers have in common the
stable determination of coefficients that are piecewise constant on a
given, that is, fixed domain partition. This partition needs to
satisfy certain geometric conditions. One can view the domain
partition as prior information which needs to be obtained by other
inverse methods.

Let us emphasize that the main effort in these papers resides in
achieving a constructive evaluation of the Lipschitz constant. This
goal requires the construction and evaluation of ad hoc singular
solutions and the use of quantitative estimates of unique
continuation, and both steps may be somewhat sophisticated due to the
presence of jumps in the coefficients. These are common themes of the
above mentioned papers, each of which however presents its own
specificity and difficulty. Also some variations from the route first
outlined in \cite{A-V} have been followed, let us mention for instance \cite{Be-Fr-V},
\cite{Be-dH-Q-S}, by taking advantage
of a general functional analytic framework developed in \cite{Bac-V}.

We wish to recall, here, that the uniqueness result
obtained by Druskin \cite{D2} was in the context of piecewise constant
conductivities too. In the present paper, we consider conductivities
that are piecewise linear instead. We note that we can adapt our
analysis to the case of piecewise linear resistivities.

In dimension $n \geq 3$ - which we consider in
geophysics - uniqueness has been established by Haberman and Tataru
for conductivities in $C^1$ \cite{Ha-T} and more recently for
Lipschitz conductivities in \cite{Ca-R}, both assuming full boundary
data. The original uniqueness result by Sylvester and Uhlmann
\cite{Sy-U} required the conductivity to be $C^{\infty}$. For the
two-dimensional case we refer to \cite{Bro-U} and the breakthrough
paper \cite{As-P} where uniqueness has been proven for conductivities
that are merely $L^{\infty}$.

The class of conductivities considered in this paper consists of
piecewise linear functions on a given domain partition, which are
possibly discontinuous at the interfaces of this partition. The
Lipschitz stability estimate we provide requires a direct proof.
(Indeed, the uniqueness result of \cite{Ca-R} ($n\ge 3$) does not apply;
in fact, in the case of partial data, the result of \cite{As-P} does
not apply either). This estimate is given in terms of the local DtoN
map. With a slight modification, our arguments can apply when the
local NtoD map is available instead, see for instance the discussion in \cite{A-G1}.

With a Lipschitz stability estimate at hand, we can apply certain
iterative methods for reconstruction within a subspace of piecewise
linear functions with a starting model at a distance less than the
radius of convergence to the unique solution \cite{dH-Q-S},
\cite{L-S}. This radius is roughly inversely
proportional to the stability constant appearing in the estimate. More
importantly, we can iteratively construct the best piecewise linear
approximation for a given domain partition. Since the stability
constant will grow at least exponentially with the number of
subdomains in the partition \cite{R}, the radius of
convergence shrinks accordingly. One can expect accurate piecewise
linear approximations with relatively few subdomains to describe the
subsurface, noting that the domain partition need not be uniform and
may show a local refinement, and hence our result provides the
necessary insight for developing a practical approach with relatively
minor prior information. Whether we can recover, also, an unknown
domain partition (such as one of tetrahedral type) is a current
subject of research.

As we mentioned earlier, the application we have in mind here is the
DC acquisition and method, which were introduced by Schlumberger in
1920 \cite{S}. Initial DC deep resistivity studies of Earth's crust
were carried out as early as in 1932 \cite{SS}. Many studies have
followed. We mention, in particular, the experiments and results by
Constable, McElhinny and McFadden \cite{Co-McE-McF} carried out in
central Australia in 1984. For a general description and the history
of the DC method (and the closely related induced electrical
polarization (IP) method) we refer to the textbooks by Koefoed
\cite{K2}, Zhdanov and Keller \cite{ZK}, and Kaufman and Anderson
\cite{Ko-An}; for a concise tutorial and review, see Ward
\cite{W}. For a finite-element method and solver for and computational
studies of the DC method, see Li and Spitzer \cite{L-S}. Here, we
consider isotropic conductivities (and therefore resistivities);
however, Earth's materials can certainly be anisotropic, which was
already recognized by Mallet and Doll \cite{M-D}. We refer to
\cite{A-G}, \cite{A-G1}, \cite{As-L-P}, \cite{B}, \cite{F-K-R},
\cite{G-L}, \cite{G-S}, \cite{L} and \cite{La-U}) for results
concerning the anisotropic case.

Through recent decades, electromagnetic methods have been widely used
in geothermal prospecting \cite{Br-Ma-V-F-Mo-E}. Amongst different
geophysical exploration methods, in geothermal prospecting,
resistivity methods have been demonstrated to be the most
effective. The reason is that the electrical resistivity of rocks is
controlled by important geothermal parameters including temperature,
fluid type and salinity, porosity, permeable pathways, fracture zones
and faults (structural), the composition of the rocks, and the
presence of alteration minerals. In this context, we mention the work
of Hersir, Bj\"{o}rnsson and Eysteinsson \cite{He-Bj-E} and, more
recently, of Fl\'{o}venz \textit{et al.} \cite{Fl}.

We briefly mention how the acquisition -- essentially probing the NtoD
map -- is carried out (see, for example, \cite{Ba1}, \cite{Ba2}). The
original acquisition was designed for two-dimensional configurations
($n=2$). The Schlumberger array consists of four collinear
electrodes. The outer two electrodes are current (boundary source)
electrodes and the inner two electrodes are the potential (receiver)
electrodes. The potential electrodes are installed at the center of
the electrode array with a small separation. The current electrodes
are gradually increased to a greater separation during the survey --
while the potential electrodes remain in the same position until the
observed voltage becomes too small to measure -- for the current to
probe deeper into the earth. Indeed, the depth resolution of the DC
method is sensitive to the separation between current electrodes
\cite{O-L}. There is also the (crossed) square-array acquisition which
is designed to be more sensitive to anisotropy than the Schlumberger
array \cite{H-W}, \cite{H}.

There are different types of electrode configuration that are commonly
used. In two-dimensional configurations, the dipole-dipole array is
widely being used because of the low electromagnetic coupling between
the current and potential circuits. In three-dimensional
configurations, the pole-pole electrode configuration is commonly
used. (In practice, the ideal pole-pole array, with only one current
and one potential electrode does not exist. To approximate the
pole-pole array, the second current and potential electrodes must be
placed at a large distance). For convenience the electrodes are
arranged in a square grid with the same unit electrode spacing in
orthogonal (coordinate) directions. (We mention the E-SCAN method
\cite{L-O}, \cite{E-O}). It can be very time-consuming to make such a
large number of measurements. To reduce the number of measurements
required without seriously degrading the resolution, ``cross-diagonal
survey'' method was introduced; here, the potential measurements are
only made at the electrodes along two orthogonal directions and the 45
degrees diagonal lines passing through the current electrode
(extracted from Loke's tutorial: 2-D and 3-D electrical imaging surveys, \cite{Lo}).

The inverse problem pertaining to resistivity interpretation was
reported as early as the 1930s (e.g. Slichter, 1933; Stevenson, 1934;
Ejen, 1938; Pekeris, 1940 \cite{Pe}). Slichter \cite{Sl}
published a method of interpretation of resistivity data over a
planarly layered earth using Hankel's Fourier-Bessel inversion
formula. It gives a unique solution if the resistivity is a continuous
function of electrode spacings. A substantial number of papers have
been written on approaches based on partial boundary data ``fitting''
or optimization to estimate the resistivity, without knowledge of
uniqueness or convergence. Narayan, Dusseault and Nobes \cite{Na-D-No}
give an extensive overview. In the context of data fitting, Parker
\cite{P} indicates and illustrates in planarly layered models the Ill-posedness of the IBVP. Interestingly, various studies and
implementations have resorted to ``blocky'' (and pseudo-layered)
representations of resistivity (\cite{Lo-A-D}, \cite{Fa-O},
\cite{Au-VC}) and, hence, fit the class of functions for which
Lipschitz stability estimates have been obtained. Finally, we mention
the complementary frequency-dependent transient electromagnetic (TEM),
magnetotelluric (MT) and electroseismic methods. The hybrid inverse
problem of electroseismic conversion was analyzed by Chen and De Hoop
\cite{Ch-dH}. The further analysis of TEM/MT \cite{Gu} is a subject of
current research.

In recent years, there has been a renewed and growing interest in the
application of electrostatic and diffuse electromagnetic inverse
boundary value problems in geophysics driven by the idea of
combining different probing fields, including acousto-elastic waves,
to identify the (poro-elastic) rock properties in Earth's interior
within a particular geological structure in an integrated
fashion. These properties certainly will not vary smoothly. We capture
the geological structure in a domain partition, let the properties be
discontinuous across subdomain boundaries of geological significance,
and approximate the parameters, here conductivity in the electrostatic
problem, in each subdomain by linear interpolation. (From a rock
physics point of view, this interpolation should be obtained from a
nonlinear upscaling, which is still an active area of research). This
approach, and the generality of these approximations, analyzed in the
context of conditional well-posedness are the novelty of this paper.

The outline of the paper is as follows. Our main assumptions and our
main result (Theorem \ref{teorema principale}) are given in section
2. Section 3 contains the proof of the main result, as well as two
intermediate results (Theorem \ref{teorema stime asintotiche} and
Proposition \ref{proposizione unique continuation finale}) needed to
build the necessary machinery. Theorem \ref{teorema stime asintotiche}
provides original asymptotic estimates for the Green's function of the
conductivity equation, its gradient and a mixed derivatives, for
conductivities that are linear on each domain $D_j$ of a given
partition $\{ D_j \}$ of $\Omega$. These asymptotic estimates are
given at the interfaces between the domains $D_j$, where the
conductivity is discontinuous. Proposition \ref{proposizione unique
  continuation finale} provides estimates of unique continuation of
the solution to the conductivity equation for piecewise linear
conductivities. Section 4 is devoted to the proof of Theorem
\ref{teorema stime asintotiche} and Proposition \ref{proposizione
  unique continuation finale}. The latter is based on the argument
introduced in \cite{A-V}[proof of Proposition 4.4], therefore only the
main differences in the two proofs are highlighted in the present
paper.

\section{Main Result}\label{sec2}
\setcounter{equation}{0}
\subsection{Notation and definitions}\label{subsec notation and definitions}

In several places within this manuscript it will be useful to single out one coordinate
direction. To this purpose, the following notations for
points $x\in \mathbb{R}^n$ will be adopted. For $n\geq 3$,
a point $x\in \mathbb{R}^n$ will be denoted by
$x=(x',x_n)$, where $x'\in\mathbb{R}^{n-1}$ and $x_n\in\mathbb{R}$.
Moreover, given a point $x\in \mathbb{R}^n$,
we will denote with $B_r(x), B_r'(x)$ the open balls in
$\mathbb{R}^{n},\mathbb{R}^{n-1}$ respectively centred at $x$ with radius $r$
and by $Q_r(x)$ the cylinder

\[Q_r(x)=B_r'(x')\times(x_n-r,x_n+r).\]

We will also denote

\begin{eqnarray*}
& & \mathbb{R}^n_+ = \{(x',x_n)\in \mathbb{R}^n| x_n>0 \};\quad\mathbb{R}^n_- = \{(x',x_n)\in \mathbb{R}^n| x_n<0 \};\\
& & B^+_r = B_r\cap\mathbb{R}^n_+;\quad B^-_r = B_r\cap\mathbb{R}^n_-;\\
& & Q^+_r = Q_r\cap\mathbb{R}^n_+;\quad Q^{-}_r = Q_r\cap\mathbb{R}^n_-.
\end{eqnarray*}

In the sequel, we will make a repeated use of quantitative
notions of smoothness for the boundaries of various domains. Let
us introduce the following notation and definitions.

\begin{definition}\label{def Lipschitz boundary}
Let $\Omega$ be a domain in $\mathbb R^n$. We say that a portion
$\Sigma$ of $\partial\Omega$ is of Lipschitz class with constants
$r_0,L$ if for any $P\in\partial\Sigma$ there exists a rigid
transformation of $\mathbb R^n$ under which we have $P=0$ and
$$\Omega\cap Q_{r_0}=\{x\in Q_{r_0}\,:\,x_n>\varphi(x')\},$$
where $\varphi$ is a Lipschitz function on $B'_{r_0}$ satisfying

\[\varphi(0)=0;\qquad
\|\varphi\|_{C^{0,1}(B'_{r_0})}\leq Lr_0.\]

It is understood that $\partial\Omega$ is of Lipschitz class with
constants $r_0,L$ as a special case of $\Sigma$, with
$\Sigma=\partial\Omega$.
\end{definition}

\begin{definition}\label{flat portion}
Let $\Omega$ be a domain in $\mathbb R^n$. We say that a portion $\Sigma$ of
$\partial\Omega$ is a flat portion of size $r_0$
if for any $P\in\Sigma$ there exists a rigid transformation of
$\mathbb R^n$ under which we have $P=0$ and



\begin{eqnarray}
\Sigma\cap{Q}_{r_{0}/3} &=&\{x\in
Q_{r_0/3}|x_n=0\}\nonumber\\
\Omega\cap {Q}_{r_{0}/3} &=&\{x\in
Q_{r_0/3}|x_n>0\}\nonumber\\
\left(\mathbb{R}^{n}\setminus\Omega\right)\cap {Q}_{r_{0}/3} &=&\{x\in
Q_{r_0/3}|x_n<0\},
\end{eqnarray}
\end{definition}


Let us rigorously define the local D-N map.\\

\begin{definition}\label{DN}
Let $\Omega$ be a domain in $\mathbb{R}^n$ with Lipschitz boundary
$\partial\Omega$ and $\Sigma$ an open non-empty (flat) open portion of
$\partial\Omega$. Let us introduce the subspace of
$H^{\frac{1}{2}}(\partial\Omega)$
\begin{equation}\label{Hco}
H^{\frac{1}{2}}_{co}(\Sigma)=\big\{f\in
H^{\frac{1}{2}}(\partial\Omega) \:\vert\:\textnormal{supp}
\:f\subset\Sigma\big\}.
\end{equation}

and its dual $H^{-\frac{1}{2}}_{co}(\Sigma)$. Assume that $\gamma\in
L^{\infty}(\Omega)$ satisfies
\begin{eqnarray}\label{ellitticita'sigma}
\lambda^{-1}\leq{\gamma}(x)\leq\lambda,
& &for\:almost\:every\:x\in\Omega,
\end{eqnarray}

then the local Dirichlet-to-Neumann map associated to $\gamma$ and
$\Sigma$ is the operator
\begin{equation}\label{mappaDN}
\Lambda_{\gamma}^{\Sigma}:H^{\frac{1}{2}}_{co}(\Sigma)\longrightarrow
{H}^{-\frac{1}{2}}_{co}(\Sigma)
\end{equation}
 defined by
\begin{equation}\label{def DN locale}
<\Lambda_{\gamma}^{\Sigma}\:g,\:\eta>\:=\:\int_{\:\Omega}
\gamma(x) \nabla{u}(x)\cdot\nabla\phi(x)\:dx,
\end{equation}
for any $g$, $\eta\in H^{\frac{1}{2}}_{co}(\Sigma)$, where
$u\in{H}^{1}(\Omega)$ is the weak solution to
\begin{displaymath}
\left\{ \begin{array}{ll} \textnormal{div}(\gamma(x)\nabla
u(x))=0, &
\textrm{$\textnormal{in}\quad\Omega$},\\
u=g, & \textrm{$\textnormal{on}\quad{\partial\Omega},$}
\end{array} \right.
\end{displaymath}
and $\phi\in H^{1}(\Omega)$ is any function such that
$\phi\vert_{\partial\Omega}=\eta$ in the trace sense. Here we
denote by $<\cdot,\:\cdot>$ the $L^{2}(\partial\Omega)$-pairing
between $H^{\frac{1}{2}}_{co}(\Sigma)$ and its dual
$H^{-\frac{1}{2}}_{co}(\Sigma)$.
\end{definition}
Note that, by \eqref{def DN locale}, it is easily verified that
$\Lambda^{\Sigma}_{\gamma}$ is selfadjoint. We will denote by
$\parallel\cdot\parallel_{*}$ the norm on the Banach space of
bounded linear operators between $H^{\frac{1}{2}}_{co}(\Sigma)$
and $H^{-\frac{1}{2}}_{co}(\Sigma)$.

\begin{remark}
Note that the space $H^{\frac{1}{2}}_{00}(\Sigma)$ (\cite{LiM}, {Chapter $1$}) is  the closure of $H^{\frac{1}{2}}_{co}(\Sigma)$ in $H^{\frac{1}{2}}(\partial\Omega)$, therefore the local DN map could be equivalently given by replacing in definition \ref{DN} the spaces $H^{\frac{1}{2}}_{co}(\Sigma)$, $H^{-\frac{1}{2}}_{co}(\Sigma)$ with $H^{\frac{1}{2}}_{00}(\Sigma)$ and its dual ${H}^{-\frac{1}{2}}_{00}(\Sigma)$ respectively  and by continuing to use the notation $<\cdot,\cdot>$ for the $L^2(\partial\Omega)$-pairing between $H^{\frac{1}{2}}_{00}(\Sigma)$ and ${H}^{-\frac{1}{2}}_{00}(\Sigma)$.

\end{remark}

\subsection{Assumptions}\label{subsection assumptions}

\subsubsection{Assumptions about the domain $\Omega$}\label{subsec assumption domain}

\begin{enumerate}

\item We assume that $\Omega$ is a domain in $\mathbb{R}^n$
satisfying

\begin{equation}\label{assumption Omega}
|\Omega|\leq N r_0 ^n,
\end{equation}

where $|\Omega|$ denotes the Lebesgue measure of $\Omega$.


\item We fix an open non-empty subset $\Sigma$ of $\partial\Omega$
(where the measurements in terms of the local D-N map are taken).

\item \[\bar\Omega = \bigcup_{j=1}^{N}\bar{D}_j,\]

where $D_j$, $j=1,\dots , N$ are known open sets of
$\mathbb{R}^n$, satisfying the conditions below.

\begin{enumerate}
\item $D_j$, $j=1,\dots , N$ are connected and pairwise
nonoverlapping polyhedrons.

\item $\partial{D}_j$, $j=1,\dots , N$ are of Lipschitz class with
constants $r_0$, $L$.

\item There exists one region, say $D_1$, such that
$\partial{D}_1\cap\Sigma$ contains a \emph{flat} portion
$\Sigma_1$ of size $r_0$ and for every $i\in\{2,\dots , N\}$ there exists $j_1,\dots ,
j_K\in\{1,\dots , N\}$ such that

\begin{equation}\label{catena dominii}
D_{j_1}=D_1,\qquad D_{j_K}=D_i.
\end{equation}

In addition we assume that, for every $k=1,\dots , K$,
$\partial{D}_{j_k}\cap \partial{D}_{j_{k-1}}$ contains a
\emph{flat} portion $\Sigma_k$ of size $r_0$ (here we agree that
$D_{j_0}=\mathbb{R}^n\setminus\Omega$), such that


\[\Sigma_k\subset\Omega,\quad\mbox{for\:every}\:k=2,\dots , K,\]

and, for every $k=1,\dots , K$, there exists $P_k\in\Sigma_k$ and
a rigid transformation of coordinates under which we have $P_k=0$
and

\begin{eqnarray}
\Sigma_k\cap{Q}_{r_{0}/3} &=&\{x\in
Q_{r_0/3}|x_n=0\}\nonumber\\
D_{j_k}\cap {Q}_{r_{0}/3} &=&\{x\in
Q_{r_0/3}|x_n>0\}\nonumber\\
D_{j_{k-1}}\cap {Q}_{r_{0}/3} &=&\{x\in
Q_{r_0/3}|x_n<0\},
\end{eqnarray}

\end{enumerate}
\end{enumerate}

\subsubsection{A-priori information on the conductivity $\gamma$}

We will consider a conductivity function $\gamma$ of type

\begin{subequations}
\begin{eqnarray}\label{a priori info su sigmaj}
& &\gamma(x)=\sum_{j=1}^{N}\gamma_{j}(x)\chi_{D_j}(x),\qquad
x\in\Omega,\label{conductivity 1}\\
& & \gamma_{j}(x)=a_j+A_j\cdot x\label{conductivity 2},
\end{eqnarray}
\end{subequations}

where $a_j\in\mathbb{R}$, $A_j\in\mathbb{R}^n$ and $D_j$, $j=1,\dots ,
N$ are the given subdomains introduced in section \ref{subsec assumption
domain}. We also assume that

\begin{equation}\label{apriorigamma}
\lambda^{-1}\le \gamma_j\le  \lambda, \qquad\textnormal{a.e\:in}\:\Omega ,\quad\mbox{for
any}\:j=1,\dots n,
\end{equation}

for some positive constant $\lambda$.


\begin{definition}
Let $N$, $r_0$, $L$, $M$, $\lambda$ be given
positive numbers with $N\in\mathbb{N}$. We will
refer to this set of numbers, along with the space dimension $n$,
as to the \textit{a-priori data}.
\end{definition}

\begin{remark}\label{remark finite dimensional space}
Observe that the class of functions of the form \eqref{conductivity 1} - \eqref{conductivity 2} is a finite dimensional linear space. The $L^{\infty}$ - norm $||\gamma||_{L^{\infty}(\Omega)}$ is equivalent to the norm

\[|||\gamma|||=\textnormal{max}_{j=1,\dots , N}\left\{|a_j|+|A_j|\right\}\]

modulo constants which only depend on the a-priori data.
\end{remark}

From now on for simplicity we will write

\[\Lambda_i = \Lambda^{\Sigma}_{\gamma^{(i)}},\qquad i=1,2.\]

\begin{theorem}\label{teorema principale}
Let $\Omega$, $D_j$, $j=1,\dots , N$ and $\Sigma$ be a domain, $N$ subdomains of $\Omega$ and a portion of $\partial\Omega$ as in section \ref{subsec assumption domain} respectively.
Let $\gamma^{(i)}$, $i=1,2$ be two conductivities satisfying \eqref{apriorigamma} and of type

\begin{equation}\label{a priori info su sigma}
\gamma^{(i)}=\sum_{j=1}^{N}\gamma^{(i)}_{j}(x)\chi_{D_j}(x),\qquad
x\in\Omega,
\end{equation}

where

\[\gamma^{(i)}_{j}(x)=a^{(i)}_j+A^{(i)}_j\cdot x,\]

with $a^{(i)}_j\in\mathbb{R}$ and $A^{(i)}_j\in\mathbb{R}^n$, then we have

\begin{equation}\label{stabilita' globale}
||\gamma^{(1)}-\gamma^{(2)}||_{L^{\infty}(\Omega)}\leq C
||\Lambda_1-\Lambda_2||_{\ast},
\end{equation}

where $C$ is a positive constant that depends on the a-priori data
only.

\end{theorem}

\begin{remark}
In this paper we are assuming that the conductivity $\gamma$ is a piecewise linear function. However, the case where the resistivity function $\rho={\gamma}^{-1}$ is piecewise linear, could be treated equally well.
\end{remark}


\section{Proof of the main result}\label{PMR}

The proof of our main result (theorem \ref{teorema principale}) is based on an argument that combines asymptotic type of estimates for the Green's function of the operator

\begin{equation}\label{operatore conduttivita' misurabile}
L=\mbox{div}\left(\gamma(x)\nabla\right)\qquad\mbox{in}\quad\Omega,
\end{equation}

(theorem \ref{teorema stime asintotiche}), with $\gamma$ satisfying \eqref{conductivity 1}-\eqref{apriorigamma}, together with a result of unique continuation (proposition \ref{proposizione unique continuation finale}) for solutions to

\[Lu=0,\qquad\mbox{in}\quad\Omega.\]

Our idea in estimating $\gamma^{(1)}-\gamma^{(2)}$  exploits, on one hand, an estimate from below of the the blow up of  some singular solutions (which we will introduce below) $S_{\mathcal{U}}$ and some of its derivatives if  $\gamma^{(1)}-\gamma^{(2)}$ is large at some point. On the other hand, we will use estimates of propagation of smallness to show that $S_{\mathcal{U}}$ needs to be small if $\Lambda_1-\Lambda_2$ is small. We will give the precise formulation of these results in what
follows.

\subsection{Singular solutions}

We will start with some general considerations about the Green's
function $G(x,y)$ associated to the operator \eqref{operatore conduttivita' misurabile}, where $\gamma$ is merely a measurable matrix valued function
satisfying the ellipticity condition \eqref{ellitticita'sigma}.

\subsubsection{Green's function}
 If $L$ is the operator given in \eqref{operatore conduttivita' misurabile}, then for every $y\in\Omega$, the Green's function $G(\cdot,y)$ is the weak solution to the Dirichlet problem

\setcounter{equation}{0}
\begin{equation}\label{GC}
\left\{
\begin{array}
{lcl} \mbox{div}(\gamma\nabla G(\cdot,y))=-\delta(\cdot - y)\ ,&
\mbox{in $\Omega$ ,}
\\
 G(\cdot,y)= 0\ ,   & \mbox{on $\partial\Omega$ ,}
\end{array}
\right.
\end{equation}

where $\delta(\cdot -y)$ is the Dirac measure at $y$. We recall
that $G$ satisfies the properties  (\cite{Lit-St-W})


\begin{equation}\label{simmetry G}
G(x,y)=G(y,x)\qquad\mbox{for}\:\mbox{every}\:x,y\in\Omega,\quad x\neq
y,
\end{equation}

\begin{eqnarray}\label{standardbeh}
0<G(x,y)< C|x-y|^{2-n}\qquad\mbox{for}\:\mbox{every}\:x,y\in\Omega,\quad
x\neq y,
\end{eqnarray}

where $C>0$ is a constant depending on $\lambda$ and $n$ only. Moreover, the following result holds true.

\begin{proposition}\label{proposizione green function}
For any $y\in \Omega$ and every $r>0$ we have that
\begin{eqnarray}\label{caccio}
\int_{\Omega\setminus B_r(y)}|\nabla G(\cdot, y)|^2\le C r^{2-n} \
\end{eqnarray}
\end{proposition}
where $C>0$ depends on $\lambda$ and $n$ only.
\begin{proof}
The proof can be obtained by combining Caccioppoli inequality with
\eqref{standardbeh} (\cite{A-V}, Proposition 3.1).
\end{proof}

\subsubsection{The $S_\mathcal{U}$ singular solutions}

Let $\gamma^{(i)}$, $i=1,2$ be two measurable
functions satisfying the ellipticity condition
\eqref{ellitticita'sigma} and let $G_{i}(x,y)$ be the Green's
functions associated to the operators

\begin{equation}\label{operatori Li}
L_i =
\mbox{div}\left(\gamma^{(i)}(x)\nabla\right)\qquad\mbox{in}\quad\Omega,\quad
i=1,2.
\end{equation}

Let $\mathcal{U}$ be an open subset of $\Omega$ and
$\mathcal{W}=\Omega\setminus\overline{\mathcal{U}}$. For any
$y,z\in\mathcal{W}$ we define

\begin{eqnarray}
S_{\mathcal{U}}({y},z)=\int_{\mathcal{U}}(\gamma^{(1)}(x)-\gamma^{(2)}(x))\nabla_x
G_1(x,{y})\cdot\nabla_x G_2(z,x)dx .
\end{eqnarray}




We recall that (see \cite{A-V}) for every $y,z\in\mathcal{W}$ we have that $S_{\mathcal{U}}(\cdot,z),
S_{\mathcal{{U}}}(y,\cdot)\in H^1_{loc}(\mathcal{W})$ are weak solutions to
\begin{eqnarray}
& &\textnormal{div} \left(\gamma^{(1)}(\cdot)\nabla
S_{\mathcal{U}}(\cdot,z)\right)=0,\quad\mbox{in}\:\mathcal{W}\\
& &\textnormal{div}
\left(\gamma^{(2)}(\cdot) \nabla S_{\mathcal{U}}(y,\cdot)\right)=0 \
,\quad\mbox{in}\:\mathcal{W}.
\end{eqnarray}

It is expected that $S_{\mathcal{{U}}}(y,z)$ blows up as $y,z$ approach simultaneously one point of $\partial\mathcal{U}$.





We will denote with
\begin{eqnarray}\label{GC2}
\Gamma(x,y)=\frac{1}{(n-2)\omega_n}|x-y|^{2-n},
\end{eqnarray}
the fundamental solution of the Laplace operator (here $\omega _n/n$ denotes the volume of the unit ball in $\mathbb{R}^n$). If $D_i$, $i=1,\dots , N$ are the domains introduced in section
\ref{subsec assumption domain} and $L$ is the operator given by
\eqref{operatore conduttivita' misurabile}, we will give asymptotic estimates for the Green's
function of $L$, with respect to
\eqref{GC2} at the interfaces between the domains $D_i$, $i=1,\dots N$.
These estimates are given below. In what follows let $G$ be the Green's function associated to
the operator $L$ in $\Omega$.


\subsubsection{Asymptotics at interfaces}\label{subsection Green function}

\begin{theorem}\label{teorema stime asintotiche}
Let $Q_{{l}+1}$ be a point such that  $Q_{{l}+1}\in B_{\frac{r_0}{8}}(P_{{l}+1})\cap \Sigma_{{l}+1}$ with $l\in\{1, \dots, N-1 \}$\ .
There exist constants $\beta, \theta, 0<\beta<1, 0<\theta<1$ and $C>0$ depending on the a priori data only such that following inequalities hold true for every $\bar{x}\in B_{\frac{r_0}{16}}(Q_{l+1})\cap
D_{j_{l+1}}$ and every $\bar{y}=Q_{l+1}-r e_n$, where $r\in
(0,\frac{r_0}{{16}})$ \

\begin{eqnarray}
& &\left|G(\bar{x},\bar{y}) -
\frac{2}{\gamma_{j_l}(Q_{l+1})+\gamma_{j_l+1}(Q_{l+1})}\Gamma(\bar{x},\bar{y})\right| \le  {C}|\bar{x}-\bar{y}|^{3-n} \label{asyfun},\\
& &\left|\nabla_x G(\bar{x},\bar{y}) -
\frac{2}{\gamma_{j_l}(Q_{l+1})+\gamma_{j_l+1}(Q_{l+1})}\nabla_x\Gamma(\bar{x},\bar{y})\right| \le  {C}|\bar{x}-\bar{y}|^{\beta+1-n}\label{asydernor},\\
& &\left|\nabla_y\nabla_xG(\bar{x},\bar{y}) -
\frac{2}{\gamma_{j_l}(Q_{l+1})+\gamma_{j_l+1}(Q_{l+1})}\nabla_y\nabla_x\Gamma(\bar{x},\bar{y})\right| \le  {C}|\bar{x}-\bar{y}|^{\theta-n} \label{asydernorgrad}\ .
\end{eqnarray}

\end{theorem}

\subsection{Quantitative unique continuation}

We recall that
up to a rigid transformation of coordinates we can assume that

\[P_1=0\qquad ;\qquad (\mathbb{R}^n\setminus\Omega)\cap B_{r_0}=\{(x^{\prime},x_n)\in B_{r_{0}}\:|\:x_n <\varphi(x^{\prime})\},\]

where $\varphi$ is a Lipschitz function such that

\[\varphi(0)=0\qquad\textnormal{and}\qquad ||\varphi||_{C^{0,1}(B_{r_{0}}^{\prime})}\leq Lr_0.\]

Denoting by

\[D_0=\left\{x\in(\mathbb{R}^n\setminus\Omega)\cap B_{r_0}\:\bigg|\:|x_i|<\frac{2}{3}r_0,\:i=1,\dots , n-1,\:\left|x_n-\frac{r_0}{6}\right|<\frac{5}{6}r_0\right\},\]

it turns out that the augmented domain $\Omega_0=\Omega\cup D_0$
is of Lipschitz class with constants $\frac{r_0}{3}$ and
$\tilde{L}$, where $\tilde{L}$ depends on $L$ only. We consider
the operator $L_i$ given by \eqref{operatori Li} and extend
$\gamma^{(i)}$ to $\tilde{\gamma}^{(i)}$ on $\Omega_0$, by
setting $\tilde{\gamma}^{(i)}|_{D_0}=1$,  for $i=1,2$. We denote by
$\tilde{G}_i$ the Green function associated to
$\tilde{L_i}=\textnormal{div}(\tilde\gamma^{(i)}(x)\nabla\cdot)$ in $\Omega_0$,
for $i=1,2$. For any number $r\in \left(0,\frac{2}{3}r_0\right)$
we also denote

\[(D_0)_r = \left\{x\in D_0\:|\:dist(x,\Omega)>r\right\}.\]

Let $K\in{1,\dots , N}$ be the subdomain of $\Omega$ such that

\begin{eqnarray}\label{DK max}
E=\|\gamma^{(1)}-\gamma^{(2)}\|_{L^{\infty}(\Omega)}=\|\gamma^{(1)}-\gamma^{(2)}\|_{L^{\infty}(D_K)}.
\end{eqnarray}


and recall that there exist $j_1,\dots , j_K\in{1,\dots , N}$ such that

\[D_{j_1}=D_1,\dots D_{j_K}=D_K,\]

with $D_{j_1},\dots D_{j_K}$ satisfying assumption $4(d)$. For simplicity, let us rearrange the indices of these subdomains so that the above mentioned chain is simply denoted by $D_1, \dots, D_K, K\le N$.  We also denote

\begin{eqnarray}
& &\mathcal{W}_k = \bigcup_{i=0}^{k}D_{i},\qquad \mathcal{U}_k = \Omega_0\setminus\overline{\mathcal{W}_k},\quad\textnormal{for}\:k=1,\dots,K\\
& &\tilde{S}_{\mathcal{U}_k}(y,z)=\int_{\mathcal{U}_k}(\tilde\gamma^{(1)}-\tilde\gamma^{(2)})\nabla\tilde{ G}_1(\cdot,y)\cdot\nabla\tilde{G}_2(\cdot,z),\quad \textnormal{for}\:k=1,\dots,K.
\end{eqnarray}

We introduce for any number $b>0$ as in \cite{A-V}, the concave
non decreasing function $\omega_{b}(t)$, defined on $(0,+\infty)$,

\begin{displaymath}
\omega_{b}(t)=\left\{ \begin{array}{ll} 2^{b}e^{-2}|\log t|^{-b},
&\quad
t\in (0,e^{-2}),\\
e^{-2}, &\quad t\in[e^{-2},+\infty)
\end{array} \right.
\end{displaymath}

and denote
\begin{eqnarray}
\omega_{b}^{(1)}=\omega,\qquad \omega_{b}^{(j)}=\omega_{b}\circ \omega_{b}^{(j-1)}.
\end{eqnarray}

The following parameters will also be introduced

\begin{eqnarray*}
& &\beta=\arctan{\frac{1}{L}},\quad\beta_1 = \arctan{\left(\frac{\sin\beta}{4}\right)},\quad\lambda_1=\frac{r_0}{1+\sin\beta_1}\\
& & \rho_1=\lambda_1\sin\beta_1,\quad a=\frac{1-\sin\beta_1}{1+\sin\beta_1}\\
& & \lambda_m=a\lambda_{m-1},\quad \rho_m = a\rho_{m-1},\quad\textnormal{for\:every}\:m\geq 2,\\
& & d_m=\lambda_m-\rho_m,\quad m\geq 1.
\end{eqnarray*}

For $k=1,\dots , K$ and a fixed point $\bar{y}\in \Sigma_{k+1}$, denote

\begin{eqnarray}
w_m(\bar{y})=\bar{y}-\lambda_m\nu(\bar{y}), \qquad \mbox{for every} \ m\ge 1 \ ,
\end{eqnarray}

where $\nu(\bar{y})$ is the exterior unit normal to $\partial D_{k}$. The following estimate for $\tilde{S}_{\mathcal{U}_k}(y,z)$ holds true, for $k=1,\dots , K$.

\begin{proposition}\label{proposizione unique continuation finale}({\bf{Estimates of unique continuation}})
If, for a positive number $\varepsilon_0$, we have

\begin{equation}\label{estim0}
\left|\tilde{S}_{\mathcal{U}_k}(y,z)\right|\leq
r_0^{2-n}\varepsilon_0,\quad for\:every\: (y,z)\in
(D_0)_{\frac{r_0}{4}}\times(D_0)_{\frac{r_0}{4}},
\end{equation}


then the following inequality holds true for every $r\in (0,d_1]$
\begin{equation}\label{estim1}
\left|\tilde{S}_{\mathcal{U}_k}\left(w_{\bar{h}}(Q_{k+1}),w_{\bar{h}}(Q_{k+1})\right)\right|
\leq
r_0^{{-n+2}}C_1^{\bar{h}}(E+\varepsilon_0)\left(\omega_{1/C}^{(2k)}\left(\frac{\varepsilon_0}{E+\varepsilon_0}\right)\right)
^{\left(1/C\right)^{\bar{h}}},
\end{equation}
\begin{equation}\label{estim2}
\left|\partial_{y_j}\partial_{z_i}\tilde{S}_{\mathcal{U}_k}\left(w_{\bar{h}}(Q_{k+1}),w_{\bar{h}}(Q_{k+1})\right)\right|
\leq
r_0^{{-n}}C_2^{\bar{h}}(E+\varepsilon_0)\left(\omega_{1/C}^{(2k)}\left(\frac{\varepsilon_0}{E+\varepsilon_0}\right)\right)
^{\left(1/C\right)^{\bar{h}}},
\end{equation}

for any $i,j=1,\dots , n$, where $Q_{k+1}\in\Sigma_{k+1}\cap B_{\frac{r_0}{8}}(P_{k+1})$,
$\bar{h}(r)=min\{m\in\mathbb{N}\:|\:d_m\leq r\}$,
$w_{\bar{h}(r)}(Q_{k+1})=Q_{k+1}-\lambda_{\bar{h}(r)}\nu(Q_{k+1})$, $\lambda_m$ has been introduced above,
$\nu$ is the exterior unit normal to $\partial{D}_k$ and $C_1,C_2>0$
depend on the a-priori data only.

\end{proposition}


\subsection{Lipschitz stability}


\begin{proof}[Proof of Theorem \ref{teorema principale}]

Let $D_K$ be the subdomain of $\Omega$ satisfying \eqref{DK max} and let $D_1,\dots D_K$ be the chain of domains satisfying assumption $4(d)$. For any $k=1,\dots , K$ we will denote by $D_{T}f$ and $\partial_\nu f$ the $n-1$ dimensional vector of the tangential partial derivatives of a function $f$ on $\Sigma_k$ and the normal partial derivative of $f$ on $\Sigma_k$ respectively. Let us also simplify our notation by replacing $\Lambda_{\gamma^{(i)}}^{\Gamma}$ with $\Lambda_i$, for $i=1,2$. We will also denote




\[\varepsilon_0=||\Lambda_1 - \Lambda_2||_{*},\qquad \delta_l=||\tilde\gamma^{(1)}-\tilde\gamma^{(2)}||_{L^{\infty}(\mathcal{W}_l)}.\]

We start our argument by estimating $\delta_1$.  By \cite{A-V} we obtain the following estimate on the first interface $\Sigma_1$ (this can also be obtained as a straightforward consequence of \cite{A-G1}[Theorem 2.3] for $\gamma^{(1)}-\gamma^{(2)}$ continuous near $\Sigma_1$)

\begin{equation}\label{lipschitz stability gamma boundary 1}
|| \tilde\gamma^{(1)}-\tilde\gamma^{(2)}||_{L^{\infty}\left(\Sigma_1\cap B_{\frac{r_0}{4}}(P_1)\right)}\leq C(\varepsilon_0+E)\left(\omega_{1/C}^{(2)}\left(\frac{\varepsilon_0}{\varepsilon_0+E}\right)\right)^
{\frac{1}{C}},
\end{equation}

where $r_0>0$ is the constant introduced in subsection \ref{subsec notation and definitions} and $C>0$ is a constant depending on the a-priori data only. Let us denote by

\begin{equation}\label{gamma linear on D1}
\alpha_1+\beta_1\cdot x = (\tilde\gamma^{(1)}_1 -\tilde\gamma^{(2)}_1)(x),
\end{equation}

and by $\{e_j\}_{j=1,\dots, n-1}$ a family of $n-1$ orthonormal vectors starting at $P_1$, defining the hyperplane containing the flat part of $\Sigma_1$. By computing $\tilde\gamma^{(1)}_1 -\tilde\gamma^{(2)}_1$ on the points $P_1$, $P_1+\frac{r_0}{5}e_j$, $j=1,\dots, n-1$, taking their differences and applying \eqref{lipschitz stability gamma boundary 1}, we obtain

\begin{eqnarray}
|\alpha_1+\beta_1\cdot P_1| &\leq & C(\varepsilon_0+E)\left(\omega_{1/C}^{(2)}\left(\frac{\varepsilon_0}{\varepsilon_0+E}\right)\right)^
{\frac{1}{C}},\label{estimate gamma at P1}\\
|\beta_1\cdot e_j| &\leq & C(\varepsilon_0+E)\left(\omega_{1/C}^{(2)}\left(\frac{\varepsilon_0}{\varepsilon_0+E}\right)\right)^
{\frac{1}{C}},\label{estimates beta along the directions ej}
\end{eqnarray}

for $j=1,\dots, n-1$, where $C>0$ is  a constant depending on the a-priori data only. To estimate $\beta_1$ along the remaining direction $\nu$ and therefore $\alpha_1$, the normal derivative

\[||\partial_{\nu}(\gamma^{(1)}-\gamma^{(2)})||_{L^{\infty}(\Sigma_1)}\]

needs to be estimated. We recall that for every $y,z\in D_0$ we have

\begin{eqnarray}\label{Alessandrini 1}
\left<(\Lambda_1 - \Lambda_2)\tilde{G}_1(\cdot,y),\tilde{G}_2(\cdot,z)\right>
&=&\int_{\Omega}(\tilde\gamma^{(1)}-\tilde\gamma^{(2)})(\cdot)\nabla\tilde{G}_1(\cdot,y)\cdot\nabla\tilde{G}_2(\cdot,z)\nonumber\\
&=&\tilde{S}_{\mathcal{U}_{0}}(y,z),
\end{eqnarray}

and

\begin{eqnarray}\label{Alessandrini 2}
\left<(\Lambda_1 - \Lambda_2)\partial_{y_n}\tilde{G}_1(\cdot,y),\partial_{z_n}\tilde{G}_2(\cdot,z)\right>
&=&\int_{\Omega}(\tilde\gamma^{(1)}-\tilde\gamma^{(2)})(\cdot)\partial_{y_n}\nabla\tilde{G}_1(\cdot,y)\cdot\partial_{z_n}\nabla\tilde{G}_2(\cdot,z)\nonumber\\
&=&\partial_{y_n}\partial_{z_n}\tilde{S}_{\mathcal{U}_{0}}(y,z).
\end{eqnarray}




From \eqref{Alessandrini 1} we obtain

\begin{eqnarray}
|\tilde{S}_{\mathcal{U}_0}(y,z)| &\leq & \varepsilon_0
||\tilde{G}_1(\cdot,y)||_{H^{1/2}_{co}(\Sigma)}||\tilde{G}_2(\cdot,z)||_{H^{1/2}_{co}(\Sigma)}\nonumber\\
&\leq & C \varepsilon_0 r_0^{2-n},\qquad
\textnormal{for\:every}\:y,z\in(D_0)_{r_0/3},
\end{eqnarray}


where $C$ depends on $A$, $L$, $\lambda$ and $n$. Let $\rho_0=\frac{r_0}{\bar{C}}$, where $\bar{C}$ is the constant introduced in Theorem \ref{teorema stime asintotiche}, let $r\in(0,d_2)$ and denote

\[w=P_1+\sigma\nu,\qquad\textnormal{where}\:\sigma=a^{\bar{h}-1}\lambda_1,\]


then

\begin{equation}\label{S=I1+I2}
\partial_{y_n}\partial_{z_n}\tilde{S}_{\mathcal{U}_0}(w,w)=I_1(w)+I_2(w),
\end{equation}

where

\[I_1(w)=\int_{B_{\rho_0}(P_1)\cap D_1}(\gamma^{(1)}-\gamma^{(2)})(\cdot)\partial_{y_n}\nabla\tilde{G}_1(\cdot,w)\cdot
\partial_{z_n}\nabla\tilde{G}_2(\cdot,w),\]

\[I_2(w)=\int_{\Omega\setminus (B_{\rho_0}(P_1)\cap
D_1)}(\gamma^{(1)}-\gamma^{(2)})(\cdot)\partial_{y_n}
\nabla\tilde{G}_1(\cdot,w)\cdot
\partial_{z_n}\nabla\tilde{G}_2(\cdot,w)\]

and (see \cite{A-V})

\begin{equation}\label{stima I2}
|I_2(w)|\leq CE\rho_0^{-n},
\end{equation}


where $C$ depends on $\lambda$ and $n$ only.  We have

\begin{eqnarray*}
|I_1(w)| & &\geq \left|\int_{B_{\rho_0}(P_1)\cap D_1}(\partial_{\nu}(\gamma^{(1)}_1-\gamma^{(2)}_1)(P_1))(x-P_1)_n\partial_{y_n}\nabla\tilde{G}_1(\cdot,w)\cdot
\partial_{z_n}\nabla\tilde{G}_2(\cdot,w)\right|\nonumber\\
& &-\int_{B_{\rho_0}(P_1)\cap D_1}|(D_T(\gamma^{(1)}_1-\gamma^{(2)}_1)(P_1))\cdot (x-P_1)'||\partial_{y_n}\nabla\tilde{G}_1(\cdot,w)|\:|
\partial_{z_n}\nabla\tilde{G}_2(\cdot,w)|\nonumber\\
& &-\int_{B_{\rho_0}(P_1)\cap D_1}|(\gamma^{(1)}_1-\gamma^{(2)}_1)(P_1)||\partial_{y_n}\nabla\tilde{G}_1(\cdot,w)|\:|
\partial_{z_n}\nabla\tilde{G}_2(\cdot,w)|.
\end{eqnarray*}

and, by Theorem \ref{teorema stime asintotiche}, this leads to

\begin{eqnarray}\label{stima S}
|I_1(w)|
& &\geq
|\partial_{\nu}(\gamma^{(1)}_1-\gamma^{(2)}_1)|C_1\int_{B_{\rho_0}(P_1)\cap
D_1}|\partial_{y_n}\nabla_x\Gamma(x,w)|^2 |x_n|\nonumber\\
& &- C_2E\int_{B_{\rho_0}(P_1)\cap
D_1}|\partial_{y_n}\nabla_x\Gamma(x,w)|\frac{|x-w|^{-n+\beta}}{\rho_0^{\beta}}|x_n|\nonumber\\
& &-C_3E\int_{B_{\rho_0}(P_k)\cap
D_1}\frac{|x-w|^{-2n+\beta}}{\rho_0^{2\beta}}|x_n|\nonumber\\
& &-\int_{B_{\rho_0}(P_1)\cap D_1}|D_T(\gamma^{(1)}_1-\gamma^{(2)}_1)|\:|(x-P_1)'|\:|\partial_{y_n}\nabla\tilde{G}_1(\cdot,w)|\:|
\partial_{z_n}\nabla\tilde{G}_2(\cdot,w)|\nonumber\\
& & -\int_{B_{\rho_0}(P_1)\cap D_1}|(\gamma^{(1)}_1-\gamma^{(2)}_1)(P_1)|\:|\partial_{y_n}\nabla\tilde{G}_1(\cdot,w)|\:|
\partial_{z_n}\nabla\tilde{G}_2(\cdot,w)|,
\end{eqnarray}

where $C_1,C_2,C_3$ are constants that depends on
$M,\lambda$ and $n$ only. Therefore, by combining
\eqref{stima S} together with \eqref{S=I1+I2} and \eqref{stima
I2}, we obtain

\begin{eqnarray*}
|I_1(w)| &\geq &
|\partial_{\nu}(\gamma^{(1)}_1-\gamma^{(2)}_1)|\tilde{C}_1\int_{B_{\rho_0}(P_1)\cap
D_1}|x-w|^{1-2n}\noindent\\
&-&\frac{\tilde{C}_2E}{{\rho_0^{\beta}}}\int_{B_{\rho_0}(P_1)\cap
D_1}|x-w|^{1-2n+\beta}\noindent\\
&-&\frac{\tilde{C}_3E}{{\rho_0^{2\beta}}}\int_{B_{\rho_0}(P_1)\cap
D_1}|x-w|^{2-2n+\beta}\nonumber\\
&-&\varepsilon_0 C_4\int_{B_{\rho_0}(P_1)\cap D_1}\:|x-w|^{1-2n}\nonumber\\
&-&\varepsilon_0 C_5\int_{B_{\rho_0}(P_1)\cap D_1}\:|x-w|^{-2n},
\end{eqnarray*}


which leads to

\begin{eqnarray}
|\partial_{\nu}(\gamma^{(1)}_1-\gamma^{(2)}_1)|\sigma^{1-n}\le |I_1(w)| + C_4\varepsilon_0 \sigma^{-n} + \tilde{C}_3 \frac{\sigma^{1-n+\beta}}{\rho_0^{\beta}},
\end{eqnarray}

where

\begin{equation}\label{I1}
|I_1(w)|\le |\partial_{y_n}\partial_{z_n}\tilde{S}_{\mathcal{U}_{0}}(w,w)| + C E \rho_0^{-n}.
\end{equation}

Thus by combining the last two inequalities we get

\begin{eqnarray}
|\partial_{\nu}(\gamma^{(1)}_1-\gamma^{(2)}_1)|\sigma^{1-n} &\le & |\partial_{y_n}\partial_{z_n}\tilde{S}_{\mathcal{U}_{0}}(w,w)|+CE\frac{\sigma^{-n+\beta}}{\rho_0^{\beta}}\nonumber\\
&+& C_4\varepsilon_0\sigma^{-n}+\tilde{C}_3\frac{\sigma^{1-n+\beta}}{\rho_0^{\beta}}
\end{eqnarray}

and by recalling that by Proposition \ref{proposizione unique continuation finale} we have

\[
\left|\partial_{y_j}\partial_{z_i}\tilde{S}_{\mathcal{U}_0}(w,w)\right|
\leq
r_0^{{-n}}C^{\bar{h}(r)}(E+\varepsilon_0)
\left(\frac{\varepsilon_0}{E+\varepsilon_0}\right)^{\left(1/C\right)^{\bar{h(r)}}},\]

we obtain

\begin{equation}
|\partial_{\nu}(\gamma^{(1)}_1-\gamma^{(2)}_1)|\le  \sigma^{-1}\left(C^{\bar{h}(r)}(E+\varepsilon_0)
\left(\frac{\varepsilon_0}{E+\varepsilon_0}\right)^{\left(1/C\right)^{\bar{h(r)}}}+C_4\varepsilon_0+CE\frac{\sigma^{\beta}}{\rho_0^{\beta}}\right)+\tilde{C}_3\frac{\sigma^{\beta}}{\rho_0^{\beta}}
\end{equation}

We need to estimate $C^{\bar{h}}$ and $\Big(\frac{1}{C}\Big)^{\bar{h}}$ in terms of $r$, where $C>1$. It turns out that

\begin{eqnarray}
C^{\bar{h}} &\leq & C^2\Big(\frac{d_1}{r}\Big)^{-\frac{1}{\log_c a}}\nonumber\\
\big(\frac{1}{C}\Big)^{\bar{h}} &\leq & \Big(\frac{1}{C}\Big)^2\Big(\frac{r}{d_1}\Big)^{-\frac{1}{\log_c a}},
\end{eqnarray}

therefore  for any $r\in(0,d_2)$

\begin{equation}\label{461}
|\partial_{\nu}(\gamma^{(1)}_1-\gamma^{(2)}_1)|\leq \left(\frac{r}{d_1}\right)C(E+\varepsilon_0)\left(\left(\frac{d_1}{r}\right)^{C}
\left(\frac{\varepsilon_0}{E+\varepsilon_0}\right)
^{\left(\frac{r}{d_1}\right)^C}+\left(\frac{r}{d_1}\right)^{\beta}\right) + \left(\frac{r}{d_1}\right)\varepsilon_0,
\end{equation}

which leads to

\begin{equation}\label{gamma on D1}
|\partial_{\nu}(\gamma^{(1)}_1-\gamma^{(2)}_1)| \leq C(\varepsilon_0+E)\left(\omega_{1/C}^{(2)}\left(\frac{\varepsilon_0}{\varepsilon_0+E}\right)\right)^{\frac{1}{C}}
\end{equation}

and thus

\begin{equation}\label{beta1 normal on D1}
|\beta_1\cdot\nu| \leq C(\varepsilon_0+E)\left(\omega_{1/C}^{(2)}\left(\frac{\varepsilon_0}{\varepsilon_0+E}\right)\right)^{\frac{1}{C}}.
\end{equation}

By combining \eqref{beta1 normal on D1} together with \eqref{gamma linear on D1}, \eqref{estimate gamma at P1}
and \eqref{estimates beta along the directions ej}, we get

\begin{equation}\label{alpha 1, beta 1}
|\alpha_1|,\:|\beta_1| \leq  C(\varepsilon_0+E)\left(\omega_{1/C}^{(2)}\left(\frac{\varepsilon_0}{\varepsilon_0+E}\right)\right)^{\frac{1}{C}}
\end{equation}

and by \eqref{alpha 1, beta 1}

\begin{equation}
\delta_1 \leq C(\varepsilon_0+E)\left(\omega_{1/C}^{(2)}\left(\frac{\varepsilon_0}{\varepsilon_0+E}\right)\right)^
{\frac{1}{C}},
\end{equation}

where $C>0$ is a constant that depends on the a-priori data only. By proceeding by induction on $l$ in order to estimate $\gamma^{(1)}_l - \gamma^{(2)}_l$ for $l=1,\dots , K$, we replace \eqref{Alessandrini 1} and \eqref{Alessandrini 2} by

\begin{eqnarray}\label{Alessandrini 1 general}
& &\left<(\Lambda_1 - \Lambda_2)\tilde{G}_1(\cdot,y),\tilde{G}_2(\cdot,z)\right>\nonumber\\
& &=\tilde{S}_{\mathcal{U}_{l-1}}(y,z)+\int_{\mathcal{W}_{l-1}}(\tilde\gamma^{(1)}-\tilde\gamma^{(2)})(\cdot)\nabla\tilde{G}_1(\cdot,y)\cdot\nabla\tilde{G}_2(\cdot,z)
\end{eqnarray}

and

\begin{eqnarray}\label{Alessandrini 2 general}
& &\left<(\Lambda_1 - \Lambda_2)\partial_{y_n}\tilde{G}_1(\cdot,y),\partial_{z_n}\tilde{G}_2(\cdot,z)\right>\nonumber\\
& &=\partial_{y_n}\partial_{z_n}\tilde{S}_{\mathcal{U}_{l-1}}(y,z)+\int_{\mathcal{W}_{l-1}}(\tilde\gamma^{(1)}-\tilde\gamma^{(2)})(\cdot)\partial_{y_n}\nabla\tilde{G}_1(\cdot,y)\cdot\partial_{z_n}\nabla\tilde{G}_2(\cdot,z)
\end{eqnarray}




respectively. By noticing that \eqref{Alessandrini 1 general} leads to (see \cite{A-V})

\begin{eqnarray}
|\tilde{S}_{\mathcal{U}_{l1}}(y,z)| &\leq & \varepsilon_0
||\tilde{G}_1(\cdot,y)||_{H^{1/2}_{co}(\Sigma)}||\tilde{G}_2(\cdot,z)||_{H^{1/2}_{co}(\Sigma)}\nonumber\\
&\leq & C (\varepsilon_0+\delta_{l-1}) r_0^{2-n},\qquad
\textnormal{for\:every}\:y,z\in(D_0)_{r_0/3},
\end{eqnarray}


where $C$ depends on $A$, $L$, $\lambda$, $n$ and by repeating the same argument applied for the special case $l=1$ and observing that

\[\delta_l\leq \delta_{l-1}+||\gamma^{(1)}_l - \gamma^{(2)}_l  ||_{L^{\infty}(D_l)},\]

we obtain (see \cite{A-V})

\[\delta_{l}\leq \delta_{l-1}+C(\varepsilon_0+\delta_{l-1}+E)\left(\omega_{1/C}^{(2(l+1))}\left(\frac{\varepsilon_0+\delta_{l-1}}{\varepsilon_0+\delta_{l-1}+E}\right)\right)^
{\frac{1}{C}},
\]

in particular for $l=K$

\[\delta_{K}\leq \delta_{K-1}+C(\varepsilon_0+\delta_{K-1}+E)\left(\omega_{1/C}^{(2(K+1))}\left(\frac{\varepsilon_0+\delta_{K-1}}{\varepsilon_0+\delta_{K-1}+E}\right)\right)^
{\frac{1}{C}},
\]

which leads to

\begin{equation*}
||\gamma^{(1)}-\gamma^{(2)}||_{L^{\infty}(\Omega)}\leq C(\varepsilon + E)\left(\omega_{\frac{1}{C}}^{(K^2)}\left(\frac{\varepsilon_0}{\varepsilon_0 + E}\right)\right)^{\frac{1}{C}},
\end{equation*}

therefore

\begin{equation}\label{463}
E\leq C(\varepsilon_0 +E)\left(\omega_{\frac{1}{C}}^{(K^2)}\left(\frac{\varepsilon_0}{\varepsilon_0 + E}\right)\right)^{\frac{1}{C}}.
\end{equation}

Assuming that $E>\varepsilon_0 e^2$ (if this is not the case then the theorem is proven) we obtain

\[E\leq C \left(\frac{E}{e^2}+E\right)\left(\omega_{\frac{1}{C}}^{(K^2)}\left(\frac{\varepsilon_0}{E}\right)\right)^{\frac{1}{C}},\]

which leads to

\[\frac{1}{C}\leq \omega_{\frac{1}{C}}^{(K^2)}\left(\frac{\varepsilon_0}{E}\right)\]

therefore

\[E\leq \frac{1}{\omega_{\frac{1}{C}}^{(-K^2)}\left(\frac{1}{C}\right)}\:\varepsilon_0,\]

where here, with a slight abuse of notation, $\omega_{\frac{1}{C}}^{(-K^2)}$ denotes the inverse function of $\omega_{\frac{1}{C}}^{(K^2)}$.

\end{proof}




\section{Proofs of Theorems \ref{teorema stime asintotiche}, \ref{proposizione unique continuation finale} }\label{PP}

\subsection{Asymptotic estimates}

\begin{theorem}\label{rego}
Let $r>0$ be a fixed number.  Let $U\in H^1(Q_r)$ be a solution to
\begin{eqnarray}
div (b(x)\nabla U)=0 \ ,
\end{eqnarray}
where
\begin{displaymath}
b(x)=\left\{ \begin{array}{ll}b^+ + B^+\cdot x,
&\quad
x\in Q^+_{r},\\
b^- + B^-\cdot x, &\quad x\in Q^-_{r} \ ,
\end{array} \right.
\end{displaymath}
where $b^+,b^-\in \mathbb{R}, B^+, B^-\in \mathbb{R}^n$ and $0<\bar{b}^{-1}\le b(x)\le \bar{b}$.
Then, there exist
positive constants $0<\alpha'\le 1, C>0$ depending on $\bar{b},r$ and $n$ only, such that for any $\rho\le \frac{r}{2}$ and
for any $x\in Q_{r-2\rho}$, the following estimate holds
\begin{eqnarray}\label{stimareg}
&&\|\nabla U\|_{L^{\infty}(Q_{\rho}(x))}+ {\rho}^{{\alpha}'}|\nabla U|_{\alpha',Q_{\rho}(x)\cap Q^+_{r} } + {\rho}^{{\alpha}'}|\nabla U|_{\alpha',Q_{\rho}(x)\cap Q^-_{r} } \nonumber \\
&&\le \frac{C}{\rho^{1+n/2}}\|U\|_{L^2(Q_{2\rho}(x))}\ ,
\end{eqnarray}

\end{theorem}

\begin{proof}
For the proof we refer to \cite[Theorem 16.2, Chap.3]{Lad-Ur}, where the
authors  obtained piecewise $C^{1,\alpha'}$
estimates for solutions to linear second order elliptic equations with
piecewise H\"{o}lder continuous coefficients and $C^{1,1}$ discontinuity interfaces (see also \cite{Li-Vo}
\cite{Li-Ni} for more recent results under weaker regularity hypothesis) .
\end{proof}

\begin{proof}[Proof of Theorem \ref{teorema stime asintotiche}]

We fix $l\in \{1, \dots, N-1 \}$. With no loss of generality we may assume that $Q_{l+1}=0$, $a_{l}=1$ and $a_{l+1}>0 $. We will denote $a_{l}=a^{-}$ and  $a_{l+1}=a^{+}$.

For any $x=(x',x_n)$ we denote $x^*=(x',-x_n)$ and we have that a fundamental solution of the operator $\mbox{div}_x(1+(a^+-1)\chi^+\nabla_x)$ has the following explicit form

\begin{equation}\label{fund}
H(x,y)= \left\{
\begin{array}
{lcl} \frac{1}{a^+}\Gamma(x,y) + \frac{a^+-1}{a^+(a^++1)}\Gamma(x,y^*)\ , &&\mbox{if}\ x_n,y_n>0\\
\frac{2}{a^++1}\Gamma(x,y) \ , &&\mbox{if}\ x_n y_n<0\ \\
\Gamma(x,y) + \frac{1-a^+}{a^++1}\Gamma(x,y^*)\ ,
&&\mbox{if}\ x_n,y_n<0.
\end{array}
\right.
\end{equation}

We then define
\begin{eqnarray}\label{diff}
R(x,y)=\tilde{ G}(x,y) - H(x,y)= \tilde{G}(x,y) - \frac{2}{a^++1}\Gamma(x,y) \ .
\end{eqnarray}


We observe that $R$ in \eqref{diff} satisfies
\begin{equation}\label{eqdiff}
\left\{
\begin{array}
{lcl} \mbox{div}_{x}(\gamma(\cdot)\nabla_{x}
R(\cdot,y))=-\mbox{div}_{x}((\gamma(\cdot)-\gamma_0(\cdot))\nabla_{x}
H(\cdot,y))\ ,& \mbox{in $\tilde{\Omega}$ ,}
\\
 R(\cdot,y)= -H(\cdot,y)\ ,   & \mbox{on $\partial\tilde{\Omega}$ .}
\end{array}
\right.
\end{equation}

By the representation formula over $\tilde{\Omega}$ we have that $R$ in
\eqref{diff} satisfies
\begin{eqnarray}
&&\  R(x,y)=\ - \int_{\tilde{\Omega}}(\gamma(\zeta)-\gamma_0(\zeta))\nabla_{\zeta}H(\zeta,y)\cdot\nabla_{\zeta}{\tilde{G}}(\zeta,x)d\zeta\ \\
&& + \int_{\partial \tilde{\Omega}}\gamma(\zeta)\partial_{\nu}\tilde{G}(\zeta,x)H(\zeta,y)d\sigma(\zeta)
.
\end{eqnarray}

We first treat the  boundary term on the right hand side of the above equation. We have that

\begin{eqnarray}
&&\left|\int_{\tilde{\Omega}}\gamma(\cdot)\partial_{\nu}\:\tilde{G}(\cdot,x)\:H(\cdot,e_ny_n)\:d\sigma \right|\\
&& \le\bar{\gamma}\:\| \partial_{\nu}\tilde{G}(\cdot,x)\|_{H^{-\frac{1}{2}}(\partial \tilde{\Omega})}\:\| H(\cdot,e_ny_n)\|_{H^{\frac{1}{2}}(\partial \tilde{\Omega})}\\
&&\le\bar{\gamma}\:\|\tilde{G}(\cdot,x)\|_{H^{1}(\tilde{\Omega}\setminus B_{r_0/2}(x))}\:\|H(\cdot,e_ny_n)\|_{H^{1}(\tilde{\Omega}\setminus B_{r_0/2}(e_ny_n))} \ .
\end{eqnarray}

Hence, we deduce that

\begin{eqnarray}\label{1R}
\left|\int_{\tilde{\Omega}}\gamma(\cdot)\partial_{\nu}\:\tilde{G}(\cdot,x)\:H(\cdot,e_ny_n)\:d\sigma \right|\le C
\end{eqnarray}

where $C>0$ is a constant depending on the a priori data only

Being $\gamma(\zeta)-\gamma_0(\zeta)$ of Lipschitz class,  we observe that the estimate

\begin{equation}\label{2aR}
|\int_{\tilde{\Omega}}(\gamma(\zeta)-\gamma_0(\zeta))\nabla_{\zeta}H(\zeta,e_ny_n)\cdot\nabla_{\zeta}{\tilde{G}}(\zeta,x)\:d\zeta| \le {C_1}|x-e_ny_n|^{3-n}
\end{equation}

can be achieved along the lines of the proof of Claim 4.3 in \cite{A-V}.

Combining \eqref{1R} and \eqref{2aR} we get
\begin{eqnarray}\label{2R}
|R(x,e_ny_n)| \le {C_1}|x-e_ny_n|^{3-n} \ ,
\end{eqnarray}
when $x\in B^+_{{r_0}}$ and $y_n\in(-{r_0},0)$.

We now focus on the estimate for $\nabla _x R(x,e_ny_n)$. Again arguing as in \cite[Claim 4.3]{A-V}, we fix $x\in B^+_{r_0/4}$ and $y_n\in(-r_0/4, 0)$ and let us denote
\begin{eqnarray}
Q=B'_{h/4}(x')\times \left (x_n, x_n +\frac{h}{4} \right)\ .
\end{eqnarray}
where $h=|x-y|$\ . We observe that  $Q\subset Q^{+}_{\frac{r_0}{2}}$. Moreover, we have that
$Q\subset Q_{\frac{h}{2}}(x)$ and $x\in \partial Q$.

By \eqref{standardbeh}, Theorem \ref{rego} and explicit computation on the behaviour of $H(x,y)$ we get

\begin{eqnarray}\label{hold}
|\nabla_x \tilde{G}(\cdot, e_ny_n)|_{\alpha',Q}\ , \ |\nabla_x H (\cdot, e_ny_n)|_{\alpha',Q} \le C  h^{-\alpha' +1-n} \ .
\end{eqnarray}

where $C>0$ is a constant depending on the a priori data only.

Hence by \eqref{diff} and \eqref{hold} we get

\begin{eqnarray}\label{holdR}
|\nabla_x R(\cdot, e_ny_n)|_{\alpha',Q}\le C  h^{-\alpha' +1-n} \ .
\end{eqnarray}

where $C>0$ is a constant depending on the a priori data only.

We recall the following interpolation inequality

\begin{equation}
\| \nabla _x R(\cdot, e_ny_n)\|_{L^{\infty}(Q)}\le C \| R(\cdot, e_ny_n)\|^{\alpha'/1+\alpha'}_{L^{\infty}(Q)}\left|\nabla_x R(\cdot, e_ny_n) \right|^{1/1+\alpha'}_{\alpha',Q} \ ,
\end{equation}

where $C>0$ is a constant depending on the a priori data only.

By the above estimate and \eqref{2R} we obtain
\begin{eqnarray}
|\nabla_x R(x,y)|\le C h ^{\beta +1-n}
\end{eqnarray}
where $\beta= \frac{\alpha'^2}{1+\alpha}$\ .

Finally, we study the behaviour of $\nabla_y\nabla_x R(x,y)$.
Let us define the cylinder $\hat{Q}= B'_{\frac{h}{8}}(0)\times \left(y_n- \frac{h}{8}, y_n \right)$. As before we have that $\hat{Q}\subset Q{^-}_{\frac{r_0}{4}},\hat{Q}\subset Q_{\frac{h}{4}}(y)$. In particular, we have that $x\notin  Q_{\frac{h}{4}}(y)$.

Let $k$ be an integer such that $k\in \{1,\dots, n\}$. We observe that $\partial_{x_k}\Gamma(x, \cdot)$ and $\partial_{x_k}G(x,\cdot)$ are solutions to
\begin{eqnarray}
&&\Delta_y (\partial_{x_k}\Gamma(x,\cdot))=0\ \ \ \mbox{in}\ \ Q_{\frac{h}{4}}(y)\ ,\\
&&\mbox{div}_y(\gamma(\cdot)\nabla_y\partial_{x_k}\tilde{G}(x,\cdot))=0\  \ \ \mbox{in}\ \ Q_{\frac{h}{4}}(y)
\end{eqnarray}
respectively.

Again by applying Theorem \ref{rego}, we have that
\begin{eqnarray}\label{1s}
|\nabla_y\partial_{x_k}\tilde{G}(x,\cdot)|_{\alpha', \hat{Q}}\le C  h^{-\alpha'-1-\frac{n}{2}}\|\partial_{x_k}\tilde{G}(x,\cdot) \|_{L^2(Q_{\frac{h}{4}}(y))}\ .
\end{eqnarray}
where $C>0$ is a constant depending on the a priori data only.

We now fix $\eta \in Q_{\frac{h}{4}}(y)$ and we notice that $\eta\notin Q_{\frac{h}{16}}(x))$. By Theorem \ref{rego} we have that
\begin{eqnarray}\label{2s}
\|\nabla_x \tilde{G}(\cdot, \eta)\|_{L^{\infty}(Q_{\frac{h}{32}}(x))}\le C h^{-1-\frac{n}{2}}\|\tilde{G}(\cdot,\eta)\|_{L^{\infty}(Q_{\frac{h}{16}}(x))}\le C h ^{1-n}
\end{eqnarray}
where $C>0$ is a constant depending on the a priori data only.

Combining \eqref{1s} and \eqref{2s} we have
\begin{eqnarray}\label{1h}
|\nabla_y\partial_{x_k}\tilde{G}(x,\cdot)|_{\alpha', \hat{Q}} \le C h^{-\alpha'-n}\ \ ,
\end{eqnarray}
where $C>0$ is a constant depending on the a priori data only.
By explicit computations we infer that
\begin{eqnarray}\label{2h}
|\nabla_y\partial_{x_k}{\Gamma}(x,\cdot)|_{\alpha', \hat{Q}} \le C h^{-\alpha'-n}\  ,
\end{eqnarray}
where $C>0$ is a constant depending on the a priori data only.
From \eqref{1h} and \eqref{2h}, we have that
\begin{eqnarray}\label{2e}
|\nabla_y\partial_{x_k}R(x,\cdot)|_{\alpha', \hat{Q}} \le C h^{-\alpha'-n}\ ,
\end{eqnarray}
where $C>0$ is a constant depending on the a priori data only.

Moreover, we observe that by analogous arguments of those discussed above, we can  infer that
\begin{eqnarray}\label{2ee}
\|\partial_{x_k}R(x,\cdot)\|_{L^{\infty}(\hat{Q})}\le C h^{\beta +1-n}
\end{eqnarray}
where $\beta= \frac{\alpha'^2}{1+\alpha}$\ .
By the following interpolation inequality

\begin{eqnarray}
\|\nabla_y\partial_{x_k}R(x,\cdot)\|_{L^{\infty}(\hat{Q})}\le C \|\partial_{x_k}R(x,\cdot) \|^{\frac{\alpha'}{\alpha'+1}}_{L^{\infty}(\hat{Q})}|\nabla_y\partial_{x_k}R(x,\cdot)|^{\frac{1}{\alpha'+1}}_{\alpha', \hat{Q}}
\end{eqnarray}
and by \eqref{2ee} and \eqref{2e} we have that
\begin{eqnarray}
|\nabla_y\partial_{x_k}R(x,y)|\le C h^{-n +\theta}
\end{eqnarray}
where $\theta=\frac{\beta\alpha'}{1+\alpha'}$\ .

\end{proof}

\subsection{Propagation of smallness}

\begin{proof}[Proof of Theorem \ref{proposizione unique continuation finale}]

By repeating the argument in \cite{A-V},[proof of Proposition 4.4] concerning a careful analysis of unique continuation argument across $K$ discontinuity interfaces and based on an iterated use of the three spheres inequality for elliptic equation, we have that for any $y,z \in B_{\rho_{\bar{h}(r)}}(w_{\bar{h}(r)}(Q_{k+1}))$

\begin{eqnarray}\label{primastima}
|{\tilde{S}}_{\mathcal{U}_k}({y},z)|\le  r_0^{{-n+2}}C^{\bar{h}(r)}(E+\varepsilon_0)\left(\omega_{1/C}^{(2k)}\left(\frac{\varepsilon_0}{E+\varepsilon_0}\right)\right)
^{\left(1/C\right)^{\bar{h}(r)}}
\end{eqnarray}

where $C>0$ is a constant depending on the a priori data only. Hence \eqref{estim1} trivially follows from \eqref{primastima}.

We now consider ${\tilde{S}}_{\mathcal{U}_k}({y},z)$ as a function of $2n$ variables where $(y,z)\in \mathbb{R}^{2n}$

Hence by \eqref{primastima} we have that

\begin{equation}\label{secstima}
|{\tilde{S}}_{\mathcal{U}_k}(y_1,\dots, y_n,z_1,\dots, z_{n})|\le r_0^{{-n+2}}C^{\bar{h}(r)}(E+\varepsilon_0)\left(\omega_{1/C}^{(2K)}\left(\frac{\varepsilon_0}{E+\varepsilon_0}\right)\right)
^{\left(1/C\right)^{\bar{h(r)}}}
\end{equation}

for any $x=(y_1,\dots, y_n,z_1,\dots, z_{n}) \in B_{\rho_{\bar{h}(r)}}(w_{\bar{h}(r)}(Q_{k+1}))\times B_{\rho_{\bar{h}(r)}}(w_{\bar{h}(r)}(Q_{k+1}))$.

Now observing that ${\tilde{S}}_{\mathcal{U}_k}(y_1,\dots, y_n,z_1,\dots, z_{n})$ is a solution in $D_k\times D_k$ of the elliptic equation

\begin{eqnarray}
{\mbox{div}}_y (\gamma^1(y){\nabla}_y {\tilde{S}}_{\mathcal{U}_k}(y,z)) + {\mbox{div}}_z (\gamma^2(z){\nabla}_z {\tilde{S}}_{\mathcal{U}_k}(y,z)) = 0
\end{eqnarray}

we have that by Schauder interior estimates that for any $i,j=1,\dots, n$ it follows

\begin{eqnarray*}
& &\|\partial_{y_i}\partial_{z_j}{\tilde{S}}_{\mathcal{U}_k}(y_1,\dots, y_n,z_1,\dots, z_{n})\|_{L^{\infty}(B_{\frac{\rho_{\bar{h}(r)}}{2}}(w_{\bar{h}(r)}(Q_{k+1}))\times B_{\frac{\rho_{\bar{h}(r)}}{2}}(w_{\bar{h}(r)}(Q_{k+1})))}\nonumber\\
& &\le\frac{C}{{\rho^2_{\bar{h}(r)-1}}}\|{\tilde{S}}_{\mathcal{U}_k }(y_1,\dots, y_n,z_1,\dots, z_{n})\|_{L^{\infty}( B_{\rho_{\bar{h}(r)}}(w_{\bar{h}(r)}(Q_{k+1}))\times B_{\rho_{\bar{h}(r)}}(w_{\bar{h}(r)}(Q_{k+1})))}
\end{eqnarray*}

Moreover, we have that being $d_{\bar{h}(r)-1}>r$,  hence it follows $r<\frac{d_0}{a\rho_0}\rho_{\bar{h}(r)}$, which in turn leads to
\begin{eqnarray}\label{Sh}
&&\|\partial_{x_i}\partial_{x_j}{\tilde{S}}_{\mathcal{U_K}}(x_1,\dots, x_{2n})\|_{L^{\infty}(\tilde{Q}_{\frac{{\rho_{\bar{h}(r)}}}{2}}(w_{\bar{h}(r)}(Q_{k+1})))} \nonumber\\
&&\le\frac{C}{r^2}\|{\tilde{S}}_{\mathcal{U_K}}(x_1,\dots, x_{2n})\|_{L^{\infty}(\tilde{Q}_{{{\rho_{\bar{h}(r)}}}}(w_{\bar{h}(r)}(Q_{k+1})))}
\end{eqnarray}

where $C>0$ is a constant depending on the a priori data only.
Noticing that
\begin{eqnarray}\label{Sh2}
\frac{\log({r/r_0})}{\log(a)}\le \bar{h}(r)\le \frac{\log({r/r_0})}{\log(a)} +1
\end{eqnarray}
we find that
\begin{eqnarray}
r^{-2}\le \left(\frac{a}{r_0}\right)^2\left(\frac{1}{a^2}\right)^{\bar{h}(r)}\ .
\end{eqnarray}

 Finally by combining \eqref{secstima}, \eqref{Sh} and the above inequality we get the desired estimate.

\end{proof}

\section*{\normalsize{Acknowledgements}}
M. V. de Hoop and R. Gaburro would like to acknowledge the support of the Isaac Newton Institute for Mathematical Sciences, Cambridge, where the research of this paper was initiated during a special semester on Inverse Problems in the Fall of 2011. The research carried out by G.Alessandrini and E.Sincich for the preparation of this paper has been supported by the grant FRA2014 ``Problemi inversi per PDE, unicit\`{a}, stabilit\`{a}, algoritmi" funded by Universit\`{a} degli Studi di Trieste.


\end{document}